\documentclass{amse-new}

\numberwithin{equation}{section} 

\begin{document}

 \PageNum{1}
 \Volume{201x}{Sep.}{x}{x}
 \OnlineTime{August 15, 201x}
 \DOI{0000000000000000}
 \EditorNote{Received 7 24, 2016, accepted x x, 201x}

\abovedisplayskip 6pt plus 2pt minus 2pt \belowdisplayskip 6pt
plus 2pt minus 2pt
\def\vsp{\vspace{1mm}}
\def\th#1{\vspace{1mm}\noindent{\bf #1}\quad}
\def\proof{\vspace{1mm}\noindent{\it Proof}\quad}
\def\no{\nonumber}
\newenvironment{prof}[1][Proof]{\noindent\textit{#1}\quad }
{\hfill $\Box$\vspace{0.7mm}}
\def\q{\quad} \def\qq{\qquad}
\allowdisplaybreaks[4]
\newtheorem{thm}{Theorem}[section]
\newtheorem{prop}[thm]{Proposition}
\newtheorem{cor}[thm]{Corollary}
\newtheorem{lem}[thm]{Lemma}
\newtheorem{rem}[thm]{Remark}
\newtheorem{defn}[thm]{Definition}
\newtheorem{que}[thm]{Question}

\newcommand{\A}{{\mathcal A}}
\newcommand{\B}{{\mathcal B}}
\newcommand{\C}{{\mathcal C}}
\newcommand{\D}{{\mathcal D}}
\newcommand{\Z}{{\mathcal Z}}
\newcommand{\E}{{\mathcal E}}
\newcommand{\F}{{\mathcal T}}
\newcommand{\I}{{\mathcal I}}
\newcommand{\M}{{\mathcal M}}
\newcommand\Rep{\operatorname{Rep}}
\newcommand{\Irr}{\operatorname{Irr}}
\newcommand\FPdim{\operatorname{FPdim}}
\newcommand\FPind{\operatorname{FPind}}
\newcommand\vect{\operatorname{Vec}}
\newcommand\SuperV{\operatorname{SuperVec}}
\newcommand\id{\operatorname{id}}
\newcommand\Tr{\operatorname{Tr}}
\newcommand\gr{\operatorname{gr}}
\newcommand\cd{\operatorname{cd}}
\newcommand\ord{\operatorname{ord}}
\newcommand\End{\operatorname{End}}
\newcommand\Alg{\operatorname{Alg}}
\newcommand\Jac{\operatorname{Jac}}
\newcommand\op{\operatorname{op}}
\newcommand\pt{\operatorname{pt}}
\newcommand\Hom{\operatorname{Hom}}
\newcommand\rk{\operatorname{rk}}
\newcommand\Pic{\operatorname{G}}
\newcommand\Fib{\operatorname{\textbf{Fib}}}


\AuthorMark{Dong J. and Zhang L. and Dai L.}                             

\TitleMark{Graded self-dual fusion categories of rank $4$}  

\title{Non-trivially graded self-dual fusion categories of rank $4$        
\footnote{Supported by the Fundamental Research Funds for the Central Universities (Grant No. KYZ201564), the Natural Science Foundation of China (Grant No. 11571173, 11201231) and the Qing Lan Project. }}                  

\author{Jingcheng \uppercase{Dong}$^*$\footnote{*Corresponding author}}             
    {College of Engineering, Nanjing Agricultural University, Nanjing 210031, China\\
    E-mail\,$:$ dongjc@njau.edu.cn }

\author{Liangyun ZHANG}     
    {College of Science, Nanjing Agricultural University, Nanjing 210095, China\\
    E-mail\,$:$ zlyun@njau.edu.cn }

\author{Li DAI}     
    {College of Engineering, Nanjing Agricultural University, Nanjing 210031, China\\
    E-mail\,$:$ daili1980@njau.edu.cn }

\maketitle%

\Abstract{Let $\C$ be a self-dual spherical fusion categories of rank $4$ with non-trivial grading. We complete the classification of Grothendieck ring $K(\C)$ of $\C$; that is, we prove that $K(\C)\cong Fib\otimes\mathbb{Z}[\mathbb{Z}_2]$, where $Fib$ is the Fibonacci fusion ring and $\mathbb{Z}[\mathbb{Z}_2]$ is the group ring on $\mathbb{Z}_2$. In particular, if $\C$ is braided then it is equivalent to $\Fib\boxtimes\vect_{\mathbb{Z}_2}^{\omega}$ as fusion categories, where $\Fib$ is a Fibonacci category and $\vect_{\mathbb{Z}_2}^{\omega}$ is a rank $2$ pointed fusion category.}      


\MRSubClass{18D10, 16T05}      

\section{Introduction}\label{sec1}
Throughout this paper we shall work over an algebraically closed field $k$ of characteristic zero.  A \emph{fusion category} over $k$ is a $k$-linear semisimple rigid tensor category  with finitely many isomorphism classes of simple objects, finite-dimensional spaces of morphisms and such that the unit object \textbf{$1$} is simple. Fusion categories arise from many areas of mathematics and physics, including representation theory of semisimple Hopf algebras and quantum groups \cite{BaKi2001lecture}, vertex operator algebras \cite{2016DongWang} and topological quantum field theory \cite{Turaer1994}.

The systematic work of classifying fusion categories of small rank dates back to Ostrik's work \cite{ostrik2003fusion}. In that paper Ostrik classified all fusion categories of rank $2$. We recall that the \emph{rank} of a fusion category is the number of isomorphism classes of its simple objects. About ten years later, Ostrik completed the classification of (pivotal) fusion categories of rank $3$ \cite{2013ostrikpivotal}. The classification of all fusion categories of rank greater than $3$ seems very difficult at the moment, only some fusion categories with additional structures were classified, see \cite{Bruillard20162364, raey,2016Bruillardrank4,Larson2014184}.

Let $\C$ be a fusion category of rank $4$. Then either exactly two simple objects are self-dual, or all four simple objects are self-dual. In \cite{Larson2014184}, Larson studied fusion categories of rank $4$ with exactly two self-dual simple objects, and gave some partial classification results. But the later case still remains open. In this paper, we study the later case and classify the Grothendieck ring of these fusion categories under the assumption that they admit non-trivial grading.

 Let $\Fib$ be a \emph{Fibonacci category}, $\vect_{\mathbb{Z}_2}^{\omega}$ be a pointed fusion category of rank $2$, where $\omega\in H^3(\mathbb{Z}_2,k^*)=\mathbb{Z}_2$.  Let $\Fib\boxtimes \vect_{\mathbb{Z}_2}^{\omega}$ be the Deligne's tensor product of $\Fib$ and $\vect_{\mathbb{Z}_2}^{\omega}$. Then $\Fib\boxtimes \vect_{\mathbb{Z}_2}^{\omega}$ is a self-dual fusion category of rank $4$.  Here, a fusion category is called self-dual if its all simple objects are self-dual. In addition, $\Fib\boxtimes \vect_{\mathbb{Z}_2}^{\omega}$ has non-trivial universal grading; more precisely, the universal grading group $\mathcal{U}(\Fib\boxtimes \vect_{\mathbb{Z}_2}^{\omega})$ of $\Fib\boxtimes \vect_{\mathbb{Z}_2}^{\omega}$ is $\mathbb{Z}_2$. Our main result is that the Grothendieck ring $K(\Fib\boxtimes \vect_{\mathbb{Z}_2}^{\omega})$ is ``unique" in the sense that the Grothendieck ring of any self-dual spherical fusion category of rank $4$  with non-trivial grading is isomorphic to $K(\Fib\boxtimes \vect_{\mathbb{Z}_2}^{\omega})$. In particular, if the fusion category considered is braided then it is equivalent to $\Fib\boxtimes \vect_{\mathbb{Z}_2}^{\omega}$ as fusion categories.

This paper is organized as follows. In Section \ref{sec2}, we give some basic definitions and results used throughout and recall Ostrik's classification result on rank $2$ fusion categories.

In Section \ref{sec3}, we study self-dual fusion categories $\C$ of rank $4$ which admit non-trivial grading. We determine their universal grading groups, fusion rules and \emph{Frobenius-Perron} (FP) dimensions of their simple objects. We obtain that the Grothendieck ring $K(\C)$ of $\C$ has two possible structures: $Fib\otimes\mathbb{Z}[\mathbb{Z}_2]$ or $K_{12}$, where $Fib$ is the Fibonacci fusion ring, $\mathbb{Z}[\mathbb{Z}_2]$ is the group ring on $\mathbb{Z}_2$, and $K_{12}$ is a new fusion ring. In particular, if the fusion category considered has the Grothendieck ring $K_{12}$ then it can not admit a structure of a braided fusion category.

In Section \ref{sec4}, we assume that the fusion categories considered are spherical. We compute the images of simple objects under the functors $F:\mathcal{Z}(\C)\to \C$ and $I:\C\to \mathcal{Z}(\C)$, where $\C$ is a self-dual fusion category of rank $4$ with non-trivial universal grading, $\mathcal{Z}(\C)$ is the Drinfeld center of $\C$, $F$ is the forgetful functor and $I$ is its right adjoint. We then use these data to prove that the Grothendieck ring of $\C$ is isomorphic to $Fib\otimes\mathbb{Z}[\mathbb{Z}_2]$. Moreover, if $\C$ is braided then it is equivalent to $\Fib\boxtimes\vect_{\mathbb{Z}_2}^{\omega}$ as fusion categories.

\section{Preliminaries and Examples}\label{sec2}
\subsection{Frobenius-Perron dimension}Let $\C$ be a fusion category and let $K(\C)$ be the Grothendieck ring of $\C$. Then the set $\Irr(\C)$ of isomorphism classes of simple objects in $\C$ is the $\mathbb{Z}^+$ basis of $K(\C)$. The FP dimension $\FPdim(X)$ of $X\in\Irr(\C)$ is the largest eigenvalue of the matrix of left multiplication by the class of $X$ in $K(\C)$. By the Frobenius-Perron Theorem $\FPdim(X)$ is a positive real number. Moreover this dimension extends to a ring homomorphism $\FPdim : K(\C) \to \mathbb{R}$ \cite[Theorem 8.6]{etingof2005fusion}. The FP dimension of $\C$ is the number $$\FPdim(\C)=\sum_{X\in\Irr(\C)}\FPdim(X)^2.$$

A simple object $X$ is called invertible if $\FPdim(X)=1$. A fusion category is called pointed if every simple object is invertible. Let $\C_{pt}$ be the fusion subcategory generated by all invertible simple objects of a fusion category $\C$. Then $\C_{pt}$ is the largest pointed fusion subcategory of $\C$.

\subsection{Fusion rules}
Let $X\in \Irr(\C)$ and $Y$ be an arbitrary object of $\C$. The multiplicity of $X$ in $Y$ is defined to be the number $[X,Y]=\dim\Hom_{\C}(X,Y)$.  So we have$$Y= \bigoplus_{X\in \Irr(\C)}[X,Y]X.$$

Let $X,Y,Z\in \Irr(\C)$. Then we have
\begin{equation}\label{multiplicity}
\begin{split}
&[X,Y]=[X^*,Y^*], \\
[X,Y\otimes Z]&=[Y^*,Z\otimes X^*]=[Y,X\otimes Z^*].
\end{split}
\end{equation}

Let $\Pic(\C)$ denote the set of isomorphism classes of invertible simple objects of $\C$. Then $\Pic(\C)$ is a group with multiplication given by tensor product. The group $\Pic(\C)$ acts on the set $\Irr(\C)$ by left tensor multiplication. Let $G[X]$ be the stabilizer of $X\in \Irr(\C)$ under this action of $\Pic(\C)$. If $g \in \Pic(\C)$ and $X,Y\in \Irr(\C)$ then
\begin{equation}\label{grouplikein1}
\begin{split}
[g,X\otimes Y]>0 \Longleftrightarrow [g,X\otimes Y]=1 \Longleftrightarrow Y= X^*\otimes g.
\end{split}
\end{equation}
In particular,
\begin{equation}\label{grouplikein2}
\begin{split}
[g,X\otimes X^{*}]>0 \Longleftrightarrow [g,X\otimes X^{*}]= 1 \Longleftrightarrow g\otimes X= X.
\end{split}
\end{equation}
Thus, for all $X\in \Irr(\C)$, we have a relation
$$X\otimes X^*=\sum_{g\in G[X]}g+\sum_{Y\in \Irr(\C)/G[X]} [Y, X\otimes X^*] Y.$$
In particular, $\textbf{1}\in X\otimes X$ if and only if $X$ is self-dual.


\subsection{Group extension of a fusion category}
Let $G$ be a finite group. A fusion category $\C$ is said to have a $G$-grading if $\C$ has a direct sum of full abelian subcategories $\C=\oplus_{g\in G}\C_g$ such that $(\C_g)^\ast=\C_{g-1}$ and $\C_g\otimes\C_h\subseteq\C_{gh}$ for all $g,h\in G$. If $\C_g\neq 0$ for all $g\in G$ then $\C=\oplus_{g\in G}\C_g$ is called a faithful $G$-grading. If this is the case $\C$ is called a $G$-extension of the trivial component $\C_e$.

By \cite[Proposition 8.20]{etingof2005fusion}, if $\C=\oplus_{g\in G}\C_g$ is a faithful grading then, for all $g,h\in G$, we have
\begin{equation}\label{FPdimgrading}
\begin{split}
\FPdim(\C_g)=\FPdim(\C_h)\,\, \mbox{and}\, \FPdim(\C)=|G| \FPdim(\C_e).
\end{split}
\end{equation}

It is known that every fusion category $\C$ has a canonical faithful grading $\C=\oplus_{g\in \mathcal{U}(\C)}\C_g$ with trivial component $\C_e=\C_{ad}$, where $\C_{ad}$ is the adjoint subcategory of $\C$ generated by simple objects in $X\otimes X^\ast$ for all $X\in \Irr(\C)$. This grading is called the universal grading of $\C$, and $\mathcal{U}(\C)$ is called the universal grading group of $\C$, see \cite{gelaki2008nilpotent}.

\subsection{Braided fusion categories}
A braided fusion category $\C$ is a fusion category admitting a braiding $c$, where the braiding is a family of natural isomorphisms: $c_{X,Y}$:$X\otimes Y\rightarrow Y\otimes X$ satisfying the hexagon axioms for all $X,Y\in\C$, see \cite{kassel1995quantum}.

Let $\D$ be a fusion subcategory of a braided fusion category $\C$. Then the M\"{u}ger centralizer $\D'$ of $\D$ in $\C$ is the fusion subcategory
$$\D'=\{Y\in\C|c_{Y,X}c_{X,Y}=\id_{X\otimes Y}\, \mbox{for all}\, X\in\D\},$$
where $c$ stands for the braiding of $\C$. The M\"{u}ger center $\mathcal{Z}_2(\C)$ of $\C$ is the M\"{u}ger centralizer $\C'$ of $\C$. The fusion category $\C$ is called non-degenerate if its M\"{u}ger center $\mathcal{Z}_2(\C)=\vect$ is trivial.

A braided fusion category is called premodular if it admits a spherical structure. For the definition of spherical structure of a fusion category, the reader is directed to \cite{etingof2005fusion}. A modular category is a non-degenerate premodular category.

\subsection{Deligne products of rank $2$ fusion categories}In this subsection, we give some examples of self-dual fusion categories of rank $4$. First we recall the Deligne's tensor product from \cite[Proposition 5.13]{deligne1990categories}. Let $\C$ and $\D$ be two fusion categories. The Deligne's tensor product $\C\boxtimes \D$ is a fusion category with simple objects $X\boxtimes Y$, where $X\in \Irr(\C)$, $Y\in \Irr(\D)$.  Let $X_1,X_2\in \Irr(\C)$ and  $Y_1,Y_2\in \Irr(\D)$. Then
$$(X_1\boxtimes Y_1)\otimes (X_2\boxtimes Y_2):=(X_1\otimes X_2)\boxtimes (Y_1\otimes Y_2).$$
The morphisms in $\C\boxtimes \D$ are defined in an obvious way.

\medbreak
Let $\C$ be a fusion category of rank $2$. Write $\Irr(\C)=\{\textbf{1},X\}$. Then the fusion rules of $\C$ are determined by a non-negative integer $n$:
$$X\otimes X= \textbf{1}\oplus nX.$$

Let $K_n$ denote the Grothendieck ring corresponding to the number $n$. The main result of \cite{ostrik2003fusion} shows that only $K_0$ and $K_1$ are categorifiable. Specifically, there exists a pointed fusion category $\vect_{\mathbb{Z}_2}^{\omega}$ such that $K(\vect_{\mathbb{Z}_2}^{\omega})=K_0$, where $\omega\in H^3(\mathbb{Z}_2,k^*)=\mathbb{Z}_2$, and there exists a Fibonacci category $\Fib$ such that $K(\Fib)=K_1$. A Fibonacci category is a rank $2$  modular category of FP dimension $\frac{5+\sqrt{5}}{2}$. It is known that Fibonacci categories fall into $2$ equivalence classes and both of them can be realized using the quantum group $U_q(sl_2)$ for $q=\sqrt[10]{1}$, see \cite{ostrik2003fusion}.

The fusion ring $K_1$ is called the Fibonacci (Yang-Lee) fusion ring and we denote it by $Fib$. It is obvious that $K_0$ is the group ring $\mathbb{Z}(\mathbb{Z}_2)$ over $\mathbb{Z}_2$.

\medbreak

Let $\A$ and $\B$ be two fusion categories of rank $2$ with $\Irr(\A)=\{\textbf{1}_{\A},A\}$ and $\Irr(\B)=\{\textbf{1}_{\B},B\}$. Then $\A\boxtimes \B$ is a self-dual fusion category with $$\Irr(\A\boxtimes \B)=\{\textbf{1}_{\A}\boxtimes\textbf{1}_{\B}, A\boxtimes\textbf{1}_{\B}, \textbf{1}_{\A}\boxtimes B, A\boxtimes B \}.$$

Write $$\textbf{1}=\textbf{1}_{\A}\boxtimes\textbf{1}_{\B},\,X=A\boxtimes\textbf{1}_{\B},\,Y=\textbf{1}_{\A}\boxtimes B,\, Z=A\boxtimes B\,\mbox{\,and\,\,} \C=\A\boxtimes \B.$$ Then $\C$ fits into three classes of fusion categories:

(1)\, $\A$ and $\B$ are both pointed. In this case, $\C$ has the fusion rules:
$$X\otimes X= Y\otimes Y= Z\otimes Z=\textbf{1}.$$
The FP dimensions of $X,Y,Z$ are all equal to $1$. In this case, $K(\C)$ is the group ring $\mathbb{Z}(\mathbb{Z}_2)\otimes\mathbb{Z}(\mathbb{Z}_2)\cong\mathbb{Z}(\mathbb{Z}_2\times\mathbb{Z}_2)$, where $\otimes$ are usual tensor product (over $\mathbb{Z}$) of rings.

(2)\, $\A$ is not pointed and $\B$ is pointed.  In this case, $\C$ has the fusion rules:
\begin{align*}
&X\otimes X= \textbf{1}\oplus X, &X\otimes Y&= Z, &X\otimes Z&= Y\oplus Z,\\
&Y\otimes Y=\textbf{1}, &Y\otimes Z&= X, &Z\otimes Z&= \textbf{1}\oplus X.
\end{align*}
The FP dimensions of $X,Y,Z$ are:
$$\FPdim(Y)=1,\FPdim(X)=\FPdim(Z)=\frac{1+\sqrt{5}}{2}.$$
In this case, $K(\C)$ is the ring $Fib\otimes \mathbb{Z}(\mathbb{Z}_2)$.

(3)\, $\A$ and $\B$ are not pointed.  In this case, $\C$ has the fusion rules:
\begin{align*}
&X\otimes X= \textbf{1}\oplus X,   &X\otimes Y&= Z, &X\otimes Z&= Y\oplus Z,\\
&Y\otimes Y= \textbf{1}\oplus Y,   &Y\otimes Z&= X\oplus Z, & Z\otimes Z&= \textbf{1}\oplus X\oplus Y\oplus Z.
\end{align*}
The FP dimensions of $X,Y,Z$ are:
$$\FPdim(X)=\FPdim(Y)=\frac{1+\sqrt{5}}{2},\FPdim(Z)=\frac{3+\sqrt{5}}{2}.$$
In this case, $K(\C)$ is the ring $Fib\otimes Fib$.

\begin{prop}\label{prop2}
Let $\C$ be the Deligne product of $2$ fusion categories of rank $2$. Then $\C$ admits a structure of a braided category.
\end{prop}

\begin{proof}
It is obvious that a pointed fusion category of rank $2$ admits a structure of a braided category. For example, the category $\Rep(\mathbb{Z}_2)$ of finite-dimensional representations of $\mathbb{Z}_2$ is braided with the standard braiding. On the other hand, a non-pointed fusion category of rank $2$ is always braided (modular) by \cite[Corollary 2.3]{ostrik2003fusion}. Hence $\C$, being Deligne product of two braided fusion categories, admits a structure of a braided category.
\end{proof}

%
%

\section{Non-trivially graded self-dual fusion category of rank $4$}\label{sec3}
In this section, we will determine all possibilities for the Grothendieck ring of a self-dual fusion category of rank 4 which has non-trivial universal grading.

\begin{lem}\label{lem2}
Let $\C$ be a self-dual fusion category. Then the universal grading group $\mathcal{U}(\C)$ is an elementary Abelian $2$-group.
\end{lem}
\begin{proof}
Let $e$ be the unit element of $\mathcal{U}(\C)$. It is clear that the unit object $\textbf{1}$ is contained in $\C_e=\C_{ad}$. Let $X$ be a simple object in some component $\C_g$. Since $X$ is self-dual, we have $[\textbf{1},X\otimes X]=[\textbf{1},X\otimes X^*]=1$. On the other hand, $X\otimes X=X\otimes X^*$ is contained in $\C_{g^2}$. This means that $\textbf{1}$ is also contained in $\C_{g^2}$. Hence we get $\C_{g^2}=\C_e$, which means that $g^2=e$. This proves that the order of any element of $\mathcal{U}(\C)$ is $1$ or $2$. Hence $\mathcal{U}(\C)$ is an elementary Abelian $2$-group.
\end{proof}

\begin{lem}\label{lem3}
Let $\C$ be a self-dual fusion category. Then $G(\C)$ is an elementary Abelian $2$-group.
\end{lem}
\begin{proof}
Let $\textbf{1}\neq X$ be an invertible object of $\C$. Since $X$ is self-dual, we have $X\otimes X= X\otimes X^*=\textbf{1}$. So every non-trivial element of $G(\C)$ has order $1$ or $2$.
\end{proof}

\medbreak

Let $\C$ be a self-dual fusion category of rank $4$. Lemma \ref{lem2} shows that the order of $\mathcal{U}(\C)$ is $1$, $2$ or $4$. In particular, if $|\mathcal{U}(\C)|=4$ then $\C$ is pointed. Since pointed fusion categories have been classified in \cite{Ostrik2003}, we only consider non-pointed fusion category in the rest of our paper.

\begin{thm}\label{thm5}
Let $\C$ be a self-dual fusion category of rank $4$. If $\C$ has non-trivial universal grading then the universal grading group $\mathcal{U}(\C)$ has order $2$ and $\C$ has the following fusion rules.

(1) $\Irr(\C_0)=\{\textbf{1},X\}$, $\Irr(\C_1)=\{Y,Z\}$, and $\C$ has the fusion rules:
\begin{align*}
&X\otimes X= \textbf{1}\oplus X, &X\otimes Y&= Z, &X\otimes Z&= Y\oplus Z,\\
&Y\otimes Y=\textbf{1}, &Y\otimes Z&= X, &Z\otimes Z&= \textbf{1}\oplus X,
\end{align*}
where $\FPdim(X)=\FPdim(Z)=\frac{1+\sqrt{5}}{2}$, $\FPdim(Y)=1$. In this case $K(\C)\cong Fib\otimes \mathbb{Z}(\mathbb{Z}_2)$.

(2) $\Irr(\C_0)=\{\textbf{1},X,Y\}$, $\Irr(\C_1)=\{Z\}$, and $\C$ has the fusion rules:
\begin{align*}
&X\otimes X=\textbf{1},&X\otimes Y&= Y,&X\otimes Z&= Z,\\
&Y\otimes Y= \textbf{1}\oplus X\oplus Y,&Y\otimes Z&= 2Z,&Z\otimes Z&= \textbf{1}\oplus X\oplus 2Y,
\end{align*}
where $\FPdim(X)=1$, $\FPdim(Y)=2,\FPdim(Z)=\sqrt{6}$. In this case we denote $K(\C)$ by $K_{12}$.
\end{thm}

\begin{proof}
Since we have assumed that $\C$ is not pointed and has non-trivial universal grading, the order of $\mathcal{U}(\C)$ is 2. Let $\C=\C_0\oplus \C_1$ be the corresponding universal grading. The ranks of $\C_0$ and $\C_1$ have two possibilities:
$$\mbox{(1)\,} |\Irr(\C_0)|=|\Irr(\C_1)|=2;\quad \mbox{(2)\,} |\Irr(\C_0)|=3,|\Irr(\C_1)|=1.$$

Case (1): Let $\Irr(\C_0)=\{1,X\},\Irr(\C_1)=\{Y,Z\}$.

Claim: $X$ is not invertible.

Proof:\quad By Lemma \ref{lem3}, the order of $G(\C)$ is $1,2$ or $4$. Since we assume that $\C$ is not pointed, the order of $G(\C)$ can not be $4$. So if $X$ is invertible then $\C_0=\C_{pt}$. This implies that $Y$ and $Z$ are non-invertible. So $\FPdim(\C_1)>\FPdim(\C_0)=2$, which contradicts equation \ref{FPdimgrading}.

Now $\C_0$ is a non-pointed fusion category with $2$ simple objects $1,X$. The fusion rules is $X\otimes X=1\oplus X$ by \cite[Theorem 2.1]{ostrik2003fusion}. Counting FP dimensions on both sides (notice that $\FPdim$ is a ring homomorphism), we have $\FPdim(X)=\frac{1+\sqrt{5}}{2}$.

Since $Y\otimes Y, Z\otimes Z, Y\otimes Z\in \C_0$, we have equations $(a,b,c\in\mathbb{Z}^+)$:
$$Y\otimes Y=\textbf{1}\oplus aX,\quad Z\otimes Z=\textbf{1}\oplus bX,\quad Y\otimes Z= cX.$$

We may reorder $Y$ and $Z$ such that $\FPdim(Y)\leq\FPdim(Z)$. This implies that $a\leq b$.

From $[X,Y\otimes Y]=[Y,X\otimes Y]=a$ and $X\otimes Y\in \C_1$, we have $X\otimes Y= aY\oplus dZ,d\in\mathbb{Z}^+$.

From $[X,Z\otimes Z]=[Z,X\otimes Z]=b, [Z,X\otimes Y])=[Z,Y\otimes X]=[Y,Z\otimes X]=d$ and $X\otimes Z\in \C_1$, we have $X\otimes Z= dY\oplus bZ$.

From $\FPdim(\C_0)=\FPdim(\C_1)$, we have $1+\FPdim(X)^2=\FPdim(Y)^2+\FPdim(Z)^2=2+(a+b)\FPdim(X)$. This equation and the fact $\FPdim(X)=\frac{1+\sqrt{5}}{2}$ imply that $a+b=1$. Since $a\leq b$, we have $a=0$ and $b=1$. Hence $Y$ is invertible and $Z\otimes Z=\textbf{1}\oplus X$. Furthermore, $Y$ being invertible implies that $X\otimes Y$ and $Y\otimes Z$ are simple objects, which means $c=d=1$, and so $Y\otimes Z= X, X\otimes Y= Z$.
This further implies that $\FPdim(Z)=\FPdim(X)=\frac{1+\sqrt{5}}{2}$. Since we have gotten $b=d=1$, we finally have $X\otimes Z= Y\oplus Z$. The fact $K(\C)\cong Fib\otimes \mathbb{Z}(\mathbb{Z}_2)$ is easy.

Case(2): Let $\Irr(\C_0)=\{\textbf{1},X,Y\},\Irr(\C_1)=\{Z\}$. We may reorder $X$ and $Y$ such that $\FPdim(X)\leq \FPdim(Y)$. Since $\FPdim(\C_0)=1+\FPdim(X)^2+\FPdim(Y)^2=\FPdim(\C_1)=\FPdim(Z)^2$, we know that $\FPdim(Z)$ is greater than $\FPdim(X)$ and $\FPdim(Y)$.

Since $X\otimes Z, Y\otimes Z\in \C_1$, we have equations $(a,b\in\mathbb{Z}^+)$:
$$X\otimes Z=aZ,\quad Y\otimes Z=bZ.$$

Counting dimensions on both sides, we get $\FPdim(X)=a,\FPdim(Y)=b$. This implies $X$ and $Y$ are integral simple objects. So $\C_0$ is an integral fusion category of rank $3$. By \cite[Lemma 2.1 and Proposition 2.2]{dong2014existence}, $G(\C_0)$ is not trivial. So $X$ must be invertible. In addition, $Y$ is not invertible by Lemma \ref{lem3}. So we have $X\otimes X=\textbf{1}, X\otimes Z= Z$. Since $Y$ is the unique non-invertible simple object in the fusion category $\C_0$, so we have $X\otimes Y= Y$.

From $[Y,X\otimes Y]=[X,Y\otimes Y]=1$, we may write $(c\in\mathbb{Z}^+)$:$Y\otimes Y=\textbf{1}\oplus X\oplus cY$. Recall that $\FPdim(Y)=b$. So we have $b^2-bc-2=0$. The fact that $b$ and $c$ are both non-negative integers show that $b=2,c=1$. Since every component of the universal grading has the same FP dimension, we have that $\FPdim(Z)=\sqrt{1+\FPdim(X)^2+\FPdim(Y)^2}=\sqrt{6}$. This completes the proof.
\end{proof}

\medbreak
Let $X$ be an object of a fusion category $\C$. We use $\langle X\rangle_{\C}$ to denote the fusion subcategory generated by $X$.
\begin{cor}\label{fusionrule2}
Let $\C$ be a self-dual fusion category of rank $4$. If $\C$ has the Grothendieck ring $K_{12}$ then it can not admit a structure of braided fusion category.
\end{cor}
\begin{proof}
 We notice that $\C$ is not symmetric since it is not integral. Keep the notations as in Theorem \ref{thm5}. Under our assumption,
$$\D_1=\langle X\rangle_{\C},\quad\D_2=\langle Y\rangle_{\C}$$
are all non-trivial fusion subcategory of $\C$. In particular, $\FPdim(\D_1)=2$ and $\FPdim(\D_2)=6$.

We first show that $\C$ is not a modular category. Suppose on the contrary that $\C$ is modular. We notices that the FP dimension of $\C$ is $12$. By \cite[Theorem 1.2 and Theorem 3.1]{Bruillard20162364}, $\C$ is equivalent to a metaplectic modular category. By the description in that paper (see the definition of a metaplectic modular category in \cite[Section 1]{Bruillard20162364}), $\C$ should have two invertible objects, two $\sqrt{3}$ -dimensional simple objects and one $2$-dimensional simple objects. This contradicts Theorem \ref{thm5} (2).

Now we suppose that $\C$ is braided. Since $\C$ is not modular, the M\"{u}ger center $\mathcal{Z}_2(\C)$ is a non-trivial symmetric fusion category. In addition,  $\mathcal{Z}_2(\C)$  is equivalent to $\D_1$ or $\D_2$ because $\D_1$ and $\D_2$ are all non-trivial fusion subcategory of $\C$. The fusion rules show that
$$X\otimes Y= Y, X\otimes Z= Z.$$
Hence we get that the fusion subcategory $\D_1$ is not equivalent to the category of super vector spaces \cite[Lemma 5.4]{muger2000galois}. This implies that the M\"{u}ger center $\mathcal{Z}_2(\C)$ must be a Tannakian subcategory. That is, $\mathcal{Z}_2(\C)$ is equivalent to the representation category $\Rep(G)$ of a finite group $G$ as a symmetric fusion category. It follows that we can form the de-equivariantization $\C_G$ of $\C$ by $\Rep(G)$. See \cite[Section 4]{drinfeld2010braided} for details on de-equivariantization.

By \cite[Remark 2.3]{etingof2011weakly}, $\C_G$ is a modular category. Combining the fact $\FPdim(\C_G)=\frac{1}{|G|}\FPdim(\C)$ with the fact $\Rep(G)= \D_1$ or $\D_2$, we get that $\FPdim(\C_G)=6$ or $2$. In both cases, $\C_G$ is a pointed fusion category \cite[Corollary 3.3]{2016DongIntegral}. It follows that $\C$ is a group-theoretical fusion category, by \cite[Theorem 7.2]{naidu2009fusion}. So $\C$ should be an integral fusion category \cite[Corollary 8.43]{etingof2005fusion}; that is, the FP dimension of any simple object of $\C$ is an integer. This is a contradiction since $\FPdim(Z)=\sqrt{6}$. This completes the proof.
\end{proof}

\section{Uniqueness of the Grothendieck ring $K(\C)$}\label{sec4}
In this section, we will use M\"{u}ger's induction functor to the Drinfeld center $\mathcal{Z}(\C)$ to show that there do not exist rank $4$ spherical fusion categories $\C$ with $K(\C)=K_{12}$.

Let $\C$ be a fusion category with $\Irr(\C)=\{\textbf{\textbf{1}}=X_0,X_1,\cdots,X_{n-1}\}$. Let $N_i$ be the matrix of left multiplication by $X_i$ in the Grothendieck ring $K(\C)$ of $\C$. The \emph{formal codegrees} of $\C$ are the eigenvalues of the matrix $N=\sum_{i=0}^{n-1}N_iN_i^*$.

Let $\mathcal{Z}(\C)$ be the Drinfeld center of a fusion category of $\C$. We use $F:\mathcal{Z}(\C)\to\C$ and $I:\C\to\mathcal{Z}(\C)$ to denote the forgetful functor and its right adjoint.

A \emph{balancing isomorphism}, or a \emph{twist}, on a braided category $\C$ is a natural automorphism $\theta:\id_{\C}\to \id_{\C}$ satisfying $\theta_1=\id_1$ and $\theta_{X\otimes Y}=(\theta_X\otimes \theta_Y)c_{Y,X}c_{X,Y}$. Let $X$ be a simple object of $\C$. Then $\theta_X$ is a scalar in $k$. In particular, any modular category is equipped with a balancing isomorphism $\theta$ such that $\theta_X$ is a root of unity for any simple object $X$ \cite{Vafa1988421}.

\begin{prop}\label{center1}
Let $\C$ be a spherical self-dual fusion category of rank $4$. Suppose that $\C$ has the Grothendieck ring $Fib\otimes \mathbb{Z}[\mathbb{Z}_2]$. Then

(1)\quad There exist non-isomorphic simple objects $\textbf{1},A,B,C$, $D_1,D_2,D_3,D_4$, $E_1,E_2,E_3,E_4$, $G_1,G_2,G_3,G_4$ in $\mathcal{Z}(\C)$ such that
\begin{align*}
&I(\textbf{1})=\textbf{1}\oplus A\oplus B\oplus C, &I(X)&= B\oplus C\oplus D_1\oplus D_2\oplus D_3\oplus D_4,\\
&I(Y)= E_1\oplus E_2\oplus E_3\oplus E_4, &I(Z)&= E_1\oplus E_2\oplus G_1\oplus G_2\oplus G_3\oplus G_4,
\end{align*}
and
\begin{align*}
&F(A)=\textbf{1},&F&(B)=\textbf{1}\oplus X, &F&(C)= \textbf{1}\oplus X,\\
&F(D_i)= X,&F&(E_1)= Y\oplus Z,&F&(E_2)= Y\oplus Z,\\
&F(E_3)= Y, &F&(E_4)= Y, &F&(G_i)= Z, i=1,2,3,4,
\end{align*}
where
\begin{align*}
&\FPdim(A)=1,&\FPdim&(B)=\frac{3+\sqrt{5}}{2}, &\FPdim&(C)= \frac{3+\sqrt{5}}{2},\\
&\FPdim(D_i)= \frac{1+\sqrt{5}}{2},&\FPdim&(E_1)= \frac{3+\sqrt{5}}{2},&\FPdim&(E_2)= \frac{3+\sqrt{5}}{2},\\
&\FPdim(E_3)= 1, &\FPdim&(E_4)= 1, &\FPdim&(G_i)= \frac{1+\sqrt{5}}{2}.
\end{align*}

(2)\quad The simple objects $\textbf{1},A,B,C$, $D_1,D_2,D_3,D_4$, $E_1,E_2,E_3,E_4$, $G_1,G_2,G_3,G_4$ are all the non-isomorphic simple objects in $\mathcal{Z}(\C)$.
\end{prop}
\begin{proof}
As in Theorem \ref{thm5}, we write $K(\C)=Fib\otimes \mathbb{Z}[\mathbb{Z}_2]=\{\textbf{1},X,Y,Z\}$.

(1)\quad Since $K(\C)$ is commutative, it has four $1$-dimensional irreducible representations. It follows from \cite[Theorem 2.13]{2013ostrikpivotal} that the object $I(\textbf{1})$ is a sum of $4$ simple objects and every object has multiplicity $1$ by \cite[Example 2.18]{2013ostrikpivotal} . It is well knows that $I(\textbf{1})$ is an algebra in $\mathcal{Z}(\C)$, and hence $\textbf{1}$ is a summand of $I(\textbf{1})$. So we can write
$$I(\textbf{1})=\textbf{1}\oplus A\oplus B\oplus C.$$

The formal codegrees of $\C$ are $5+\sqrt{5},5+\sqrt{5},5-\sqrt{5},5-\sqrt{5}$. Again by \cite[Theorem 2.13]{2013ostrikpivotal}, we have
$$\FPdim(A)=1,\FPdim(B)=\FPdim(C)=\frac{3+\sqrt{5}}{2}.$$

By \cite[Theorem 2.5]{2013ostrikpivotal}, we have
\begin{equation*}
\begin{split}
5+\sqrt{5}&=\FPdim(\C)=Tr(\theta_{I(\textbf{1})})\\
&=1+\FPdim(A)\theta_A+\FPdim(B)\theta_B+ \FPdim(C)\theta_C\\
&=1+\theta_A+\frac{3+\sqrt{5}}{2}\theta_B+ \frac{3+\sqrt{5}}{2}\theta_C.
\end{split}
\end{equation*}

Because $\theta_V$ is a root of unity for any simple object $V$ in $\mathcal{Z}(\C)$, we must have $\theta_A=\theta_B=\theta_C=1$.

By \cite[Proposition 5.4]{etingof2005fusion}, we have
\begin{equation*}
\begin{split}
F(I(\textbf{1}))&= F(\textbf{1})\oplus F(A)\oplus F(B)\oplus F(C)\\
&=\bigoplus_{T\in\Irr(\C)}T\otimes\textbf{1}\otimes T^*\\
&=4\cdot\textbf{1}\oplus 2X.
\end{split}
\end{equation*}
So we must have
$$F(A)=\textbf{1},F(B)=\textbf{1}\oplus X, F(C)= \textbf{1}\oplus X.$$

\medbreak
Again by \cite[Proposition 5.4]{etingof2005fusion}, we have
$$F(I(X))= \bigoplus_{T\in\Irr(\C)}T\otimes X\otimes T^*=2\cdot\textbf{1}\oplus 6X.$$

Since $I$ and $F$ are adjoint, we have
\begin{equation*}
\begin{split}
\dim\Hom(I(1),I(X))&=\dim\Hom(F(I(1)),X)=2,\\
\dim\Hom(B,I(X))&=\dim\Hom(F(B),X)=1,\\
\dim\Hom(C,I(X))&=\dim\Hom(F(C),X)=1,\\
\dim\Hom(I(X),I(X))&=\dim\Hom(F(I(X)),X)=6.
\end{split}
\end{equation*}

The first equality implies that $I(1)$ and $I(X)$ have $2$ common simple objects; the second and the third equalities imply that the multiplicities of $B$ and $C$ in $I(X)$ are $1$. So we may write $I(X)=B\oplus C\oplus W$ for some objects $W$. The last equality implies we have two possibilities:
$$I(X)=B\oplus C\oplus 2D\, \mbox{\quad or\quad} I(X)=B\oplus C\oplus D_1\oplus D_2\oplus D_3\oplus D_4,$$
where $D,D_1, D_2, D_3,  D_4$ are simple objects of $\mathcal{Z}(\C)$ and the last four are non-isomorphic.

If $I(X)=B\oplus C\oplus 2D$ then
$$2\cdot \textbf{1}\oplus 2X\oplus 2F(D)= F(I(X))=2\cdot \textbf{1}\oplus 6X.$$
This implies that $F(D)=2X$, and hence $\FPdim(D)=1+\sqrt{5}$. By \cite[Theorem 2.5]{2013ostrikpivotal}, we have
$$0=Tr(\theta_{I(X)})=\FPdim(B)\theta_B+ \FPdim(C)\theta_C+2\FPdim(D)\theta_D.$$
This implies that $\theta_D=-\frac{1+\sqrt{5}}{4}$, which is a contradiction since $\theta_D$ must be a root of unity. Therefore, $I(X)=B\oplus C\oplus D_1\oplus D_2\oplus D_3\oplus D_4$, and hence $F(D_i)=X$, $\FPdim(D_i)=\frac{1+\sqrt{5}}{2}$.

\medbreak
By \cite[Proposition 5.4]{etingof2005fusion}, we have
$$F(I(Y))= \bigoplus_{T\in\Irr(\C)}T\otimes Y\otimes T^*=4Y\oplus 2Z\mbox{\quad and\quad}$$

\begin{equation*}
\begin{split}
\dim\Hom(I(1),I(Y))&=\dim\Hom(F(I(1)),Y)=0,\\
\dim\Hom(I(X),I(Y))&=\dim\Hom(F(I(X)),Y)=0,\\
\dim\Hom(I(Y),I(Y))&=\dim\Hom(F(I(Y)),Y)=4.
\end{split}
\end{equation*}

So we can write
$$I(Y)=2E \mbox{\quad or\quad} I(Y)=E_1\oplus E_2\oplus E_3\oplus E_4,$$
where $E_1, E_2, E_3,  E_4$ are non-isomorphic simple objects of $\mathcal{Z}(\C)$.

If $I(X)=2E$ then
$$2F(E)= F(I(Y))= 4Y\oplus 2Z.$$
This implies that $F(E)=2Y\oplus Z$, and hence $\FPdim(E)=\frac{5+\sqrt{5}}{2}$. By \cite[Theorem 2.5]{2013ostrikpivotal}, we have
$$0=Tr(\theta_{I(Y)})=2\FPdim(E)\theta_E.$$
This implies that $\theta_E=0$, which is a contradiction since $\theta_D$ is a root of unity. Therefore, $I(Y)=E_1\oplus E_2\oplus E_3\oplus E_4$. We hence have
$$4Y\oplus 2Z= F(I(Y))= F(E_1)\oplus F(E_2)\oplus F(E_3)\oplus F(E_4).$$
From $\dim\Hom(E_i,I(Y))=\dim\Hom(F(E_i),Y)=1$, we may write
\begin{equation}\label{F(Ei)}
\begin{split}
F(E_i)=Y\oplus a_iZ, \mbox{\,where\,\,} a_1+a_2+a_3+a_4=2.
\end{split}
\end{equation}

\medbreak
By \cite[Proposition 5.4]{etingof2005fusion}, we have
$$F(I(Z))= \bigoplus_{T\in\Irr(\C)}T\otimes Z\otimes T^*=2Y\oplus 6Z \mbox{\quad and\quad}$$

\begin{equation*}
\begin{split}
\dim\Hom(I(1),I(Z))&=\dim\Hom(F(I(1)),Z)=0,\\
\dim\Hom(I(X),I(Z))&=\dim\Hom(F(I(X)),Z)=0,\\
\dim\Hom(I(Y),I(Z))&=\dim\Hom(F(I(Y)),Z)=2,\\
\dim\Hom(I(Z),I(Z))&=\dim\Hom(F(I(Z)),Z)=6.
\end{split}
\end{equation*}
So we know that the decomposition of $I(Y)$ and $I(Z)$ have two simple objects in common, say $E_1$ and $E_2$, hence we can write
$$I(Z)=E_1\oplus E_2\oplus 2G\, \mbox{\quad or\quad} I(Z)=E_1\oplus E_2\oplus G_1\oplus G_2\oplus G_3\oplus G_4,$$
where $G_1, G_2, G_3,  G_4$ are non-isomorphic simple objects of $\mathcal{Z}(\C)$.

In both cases, we have $\dim\Hom(E_i,I(Z))=\dim\Hom(F(E_i),Z)=1$ for $i=1,2$. Together with equation (\ref{F(Ei)}), we have
$$F(E_1)= F(E_2)= Y\oplus Z,\,  F(E_3)= F(E_4)= Y.$$

If $I(Z)=E_1\oplus E_2\oplus 2G$ then
$$F(E_1)\oplus F(E_2)\oplus 2F(G)= F(I(Z))= 2Y\oplus 6Z.$$
This implies that $F(G)=2Z$, and hence $\FPdim(G)=1+\sqrt{5}$. By \cite[Theorem 2.5]{2013ostrikpivotal}, we have
$$0=Tr(\theta_{I(Z)})=\FPdim(E_1)\theta_{E_1}+ \FPdim(E_2)\theta_{E_1}+2\FPdim(G)\theta_G.$$
This implies that $\theta_G=-\frac{1+\sqrt{5}}{8}(\theta_{E_1}+\theta_{E_2})$. Taking the absolute value of both sides, we get $|\theta_G|<\frac{1}{2}\cdot|\theta_{E_1}+\theta_{E_2}|<1$, which is a contradiction since $\theta_G$ is a root of unity. Therefore, $I(Z)=E_1\oplus E_2\oplus G_1\oplus G_2\oplus G_3\oplus G_4$. So we have
$$F(E_1)\oplus F(E_2)\oplus F(G_1)\oplus F(G_2)\oplus F(G_3)\oplus F(G_4)= F(I(Z))= 2Y\oplus 6Z.$$
This implies that $F(G_i)= Z$, and hence $\FPdim(G_i)=\frac{1+\sqrt{5}}{2}$ for $i=1,2,3,4$.

\medbreak
(2)\quad This part follows from the fact
\begin{equation*}
\begin{split}
&1+\FPdim(A)^2+\FPdim(B)^2+\FPdim(C)^2+\\
&\sum_{i=1}^4\FPdim(D_i)^2+\sum_{i=1}^4\FPdim(E_i)^2+\sum_{i=1}^4\FPdim(G_i)^2\\
&=30+10\sqrt{5}=\FPdim(\C)^2=\FPdim(\mathcal{Z}(\C)).
\end{split}
\end{equation*}
\end{proof}

\begin{prop}\label{sphfusion2}
There do not exist rank $4$ spherical self-dual fusion categories $\C$ with $K(\C)=K_{12}$. That is, the fusion ring $K_{12}$ is not categorifiable.
\end{prop}
\begin{proof}
Let $\C$ be a spherical self-dual fusion category of rank $4$. Assume that $K(\C)=K_{12}$. We will follow the line of proof of Proposition \ref{center1} to find a contradiction. As in Theorem \ref{thm5}, we write $K(\C)=\{\textbf{1},X,Y,Z\}$.

The object $I(\textbf{1})$ can be decomposed as follows:
$$I(\textbf{1})=\textbf{1}\oplus A_1\oplus A_2\oplus A_3,$$
where $A_1,A_2,A_3$ are non-isomorphic simple objects in $\mathcal{Z}(\C)$.

The formal codegrees of $\C$ are $12,12,3$ and $2$. By \cite[Theorem 2.13]{2013ostrikpivotal}, we have
$$\FPdim(A_1)=1,\FPdim(A_2)=4,\FPdim(A_3)=6.$$

By \cite[Theorem 2.5]{2013ostrikpivotal}, we have
$$\FPdim(\C)=Tr(\theta_{I(\textbf{1})})=1+\FPdim(A_1)\theta_{A_1}+\FPdim(A_2)\theta_{A_2}+ \FPdim(A_3)\theta_{A_3}.$$
Because $\theta_V$ is a root of unity for any simple object $V$ in $\mathcal{Z}(\C)$, we have $\theta_{A_1}=\theta_{A_2}=\theta_{A_3}=1$.

By \cite[Proposition 5.4]{etingof2005fusion}, we have
\begin{equation}\label{equation1}
\begin{split}
F(I(\textbf{1}))&= F(\textbf{1})\oplus F(A_1)\oplus F(A_2)\oplus F(A_3)\\
&=\bigoplus_{T\in\Irr(\C)}T\otimes\textbf{1}\otimes T^*\\
&=4\cdot\textbf{1}\oplus 2X\oplus3Y.
\end{split}
\end{equation}

Let $F(A_i)=\textbf{1}\oplus a_iX\oplus b_iY$. Counting FP dimensions on both sides, we get
$$a_1=b_1=0,a_2+2b_2=3,a_3+2b_3=5.$$

Together wit equation (\ref{equation1}), we also get $a_2+a_3=2,b_2+b_3=3$. So $a_2=a_3=1,b_2=1,b_3=2$, and hence $F(A_1)=\textbf{1}$, $F(A_2)=\textbf{1}\oplus X\oplus Y$, $F(A_3)= \textbf{1}\oplus X\oplus 2Y$.

By \cite[Proposition 5.4]{etingof2005fusion}, we have
$$F(I(X))= \bigoplus_{T\in\Irr(\C)}T\otimes X\otimes T^*=2\cdot\textbf{1}\oplus 4X\oplus 3Y.$$

Since $I$ and $F$ are adjoint, we have
\begin{equation*}
\begin{split}
\dim\Hom(I(1),I(X))&=\dim\Hom(F(I(1)),X)=2,\\
\dim\Hom(I(X),I(X))&=\dim\Hom(F(I(X)),X)=4,\\
\dim\Hom(A_2,I(X))&=\dim\Hom(F(B),X)=1,\\
\dim\Hom(A_3,I(X))&=\dim\Hom(F(C),X)=1.
\end{split}
\end{equation*}

So we may write
$$I(X)= A_2\oplus A_3\oplus D_1\oplus D_2,$$
where $D_1, D_2$ are non-isomorphic simple objects of $\mathcal{Z}(\C)$. Hence we have
\begin{equation}\label{equation2}
\begin{split}
F(I(X))&= F(A_2)\oplus F(A_3)\oplus F(D_1)\oplus F(D_2)\\
&= 2\cdot\textbf{1}\oplus 4X\oplus 3Y.
\end{split}
\end{equation}
This implies that $F(D_1)= F(D_2)= X$. Hence $\FPdim(D_1)=\FPdim(D_1)=1$.

By \cite[Theorem 2.5]{2013ostrikpivotal}, we have
\begin{equation*}
\begin{split}
0=Tr(\theta_{I(X)})=\sum_{i=2}^3\FPdim(A_i)\theta_{A_i}+\sum_{i=1}^2\FPdim(D_i)\theta_{D_i}.
\end{split}
\end{equation*}

This implies that $\theta_{D_1}+ \theta_{D_2}=-10$, which is a contradiction since $\theta_{D_1}$ and $\theta_{D_2}$ are roots of unity. This proves the proposition.
\end{proof}

\medbreak
\begin{thm}\label{main}
Let $\C$ be a spherical self-dual fusion category of rank $4$ with non-trivial universal grading. Then $K(\C)\cong Fib\otimes\mathbb{Z}[\mathbb{Z}_2]$.
\end{thm}
\begin{proof}
By Theorem \ref{thm5}, $K(\C)\cong Fib\otimes\mathbb{Z}[\mathbb{Z}_2]$ or $K_{12}$, while Proposition \ref{sphfusion2} shows that $K(\C)\ncong K_{12}$, hence $K(\C)\cong Fib\otimes\mathbb{Z}[\mathbb{Z}_2]$.
\end{proof}

\begin{cor}
Let $\C$ be a spherical self-dual braided fusion category of rank $4$ with non-trivial universal grading. Then $\C\cong \Fib\boxtimes\vect_{\mathbb{Z}_2}^{\omega}$ as fusion categories.
\end{cor}
\begin{proof}
Since $\C$ is braided, Theorem \ref{thm5} and Corollary \ref{fusionrule2} show that $K(\C)\cong Fib\otimes \mathbb{Z}(\mathbb{Z}_2)$. Keep the notations as in Theorem \ref{thm5}. Let $\D$ be the fusion subcategory of $\C$ generated by $X$. Then $\D$ is of rank $2$ and not pointed, hence $\D\cong\Fib$ is a modular category \cite[Corollary 2.3]{ostrik2003fusion}. By M\"{u}ger's Theorem \cite[Theorem 4.2]{muger2003structure}, $\C$ is equivalent to $\Fib\boxtimes \Fib'$, where $\Fib'$ is the M\"{u}ger centralizer of $\Fib$. Obviously, $\Fib'$ is pointed and hence $\Fib'\cong \vect_{\mathbb{Z}_2}^{\omega}$ for some $\omega\in H^3(\mathbb{Z}_2,k^*)$. This proves the corollary.
\end{proof}

\acknowledgements{\rm The authors appreciate the two referees for their suggestions and comments. The first author is happy to thank Henry Tucker for useful discussion. }



\end{document}